\documentclass{amsart}
\usepackage{graphicx,tikz, caption}
\usepackage[utf8]{inputenc}
\usepackage{hyperref}
\usepackage[normalem]{ulem}
\usepackage{amsmath}
\usepackage{amssymb}
\usepackage{amsthm}
\usepackage{mathtools}
\usepackage{verbatim}
\usepackage{xcolor}
\usepackage{fullpage}
\usepackage[skip=3pt, indent=20pt]{parskip}
\usepackage{caption}
\usepackage{comment}
\usepackage{marginnote}
\usepackage{enumitem}
\usepackage{forest} 
\usepackage[
backend=biber,
style=alphabetic,
sorting=nty
]{biblatex}
\newcommand{\R}{\mathbb{R}}

\newcommand{\Z}{\mathbb{Z}}
\newtheorem{theorem}{Theorem}
\newtheorem*{thm}{Theorem}

\newtheorem{lemma}{Lemma}

\newtheorem{ex}{Example}
\newtheorem*{question}{Question}

\begin{document}

\title[Monotonicity formulas for harmonic functions on the infinite regular tree]{Monotonicity formulas for harmonic functions\\ on the infinite regular tree}

\author[KA]{Kathryn Atwood}
\address{Department of Mathematics, Western Washington University, Bellingham, WA 98225}
\email{atwoodk2@wwu.edu}

\author[MSVG]{Mariana Smit Vega Garcia}
\address{Department of Mathematics, Western Washington University, Bellingham, WA 98225}
\email{smitvem@wwu.edu}

\author[RW]{Richard Wang}
\address{Department of Mathematics, Western Washington University, Bellingham, WA 98225}
\email{wangr6@wwu.edu}

\date{}

\begin{abstract}
We continue the program initiated in \cite{SVGS}. In this paper, we focus on the infinite $d-$regular tree, and prove the monotonicity of a weighted Dirichlet energy, a Weiss-type monotonicity formula, and a generalization of the Almgren monotonicity formula of \cite{SVGS} for $p\ge 1$. We also compute examples in the infinite $2-$ and $3-$regular trees.
\end{abstract}

\maketitle

\section{Introduction}
We will work with the infinite $d-$regular tree $T_d = (V, E)$ where $d \geq 2$, which is the unique tree where each vertex has exactly $d$ neighbors. We fix a root $x_0$ and denote $V_k=\{x\in V \ : \ d(x_0,x)=k\}$. Recall that a function $u:V \rightarrow \mathbb{R}$ is harmonic if, for each vertex $x \in V$,
\[
\sum_{(x, y) \in E} (u(x) - u(y)) = 0 \qquad \mbox{or, equivalently} \qquad  u(x) = \frac{1}{d} \sum_{(x,y) \in E} u(y).
\]

The natural analog of a gradient $|\nabla u|$ along an edge $(a,b) \in E$ is simply $|u(a) - u(b)|$, and $p-$th powers $|\nabla u|^p$ should be understood in the same way. In the $d-$regular tree, any edge $(a,b) \in E$ has the property that the two endpoints have a different distance from $x$ (a vertex in $V$), thus, after possibly relabeling $a$ and $b$, we have $d(a,x) = d(b,x)+1$. Consequently, we can define a notion of distance between an edge $(a,b) = e \in E$ and a vertex $v \in V$ via
\[
d(v,e) = \min\left\{ d(x,a), d(x,b) \right\}.
\]

The interested reader is pointed to \cite{zbMATH00980775, zbMATH06946113, zbMATH03050722, zbMATH02133156} and references therein. 
Recently, discrete harmonic functions have been the focus of renewed interest. In \cite{zbMATH07542899}, the authors construct discrete harmonic functions in wedges. In \cite{zbMATH08096041}, unique continuation on planar graphs is addressed. An interesting three ball theorem for discrete harmonic functions is proved in 
\cite{zbMATH06375597}. The paper \cite{zbMATH07513362} proves that a discrete harmonic function that is bounded in large portions of $\Z^2$ must be constant. We also mention \cite{zbMATH07763395}, where the author proves uniqueness results for solutions of continuous and discrete PDEs, and \cite{zbMATH07160155}, where the authors discuss localization on the two dimensional lattice.
Recently, one of us and Steinerberger proved in \cite{SVGS} a version of the Almgren monotonicity formula for discrete harmonic functions. Continuing that program, in this paper, we prove extensions of three celebrated monotonicity formulas to the setting of discrete
harmonic functions on the $d-$regular tree.
Our first main result, Theorem \ref{T:Dirichlet} below, can be thought of as a generalization of the fact that for ``standard'' harmonic functions in $\R^n$, the function
\begin{equation}\label{doesntwork}
r\mapsto \frac{1}{r^n}\int_{B_r(x_0)}|\nabla u|^2
\end{equation}
is monotonic non-decreasing. The analogue of $\int_{B_r(x_0)}|\nabla u|^2$ for a harmonic function on the $d-$regular tree is
\[
\sum_{\substack{e=(y,z)\in E \\d(e,x_0)< k}}|u(y)-u(z)|^2=
\sum_{\ell=0}^{k-1}\sum_{\substack{(y,z)\in E\\ y\in V_{\ell}, \ z\in V_{\ell+1}}}|u(y)-u(z)|^2.
\]
In the context of discrete harmonic functions, however, the function 
\[
k\mapsto \sum_{\ell=0}^{k-1}\sum_{\substack{(y,z)\in E\\ y\in V_{\ell},\ z\in V_{\ell+1}}}|u(y)-u(z)|^2,
\]
need not diverge to infinity for a nonconstant $u$ (see \S \ref{S:Dirichlet}), which means weights need to be included for a version of \eqref{doesntwork} to hold. This is done in Theorem \ref{T:Dirichlet}.

Our second main results are Weiss type monotonicity formulas (Theorems \ref{T:Weiss}, \ref{T:Weiss2d}) for harmonic functions in the infinite $d-$regular tree. Both of these results were inspired by those originally obtained by Weiss in \cite{Wei98} and \cite{Wei99} in the context of the classical obstacle problem. These monotonicity formulas play a fundamental role in the study of numerous free boundary problems, including the thin obstacle problem (see, for example, \cite{zbMATH06194692, zbMATH06444642, zbMATH05592778, zbMATH07019848, zbMATH06706425, zbMATH06578536, zbMATH07144967}).

Finally, we return to the Almgren monotonicity formula proved in \cite{SVGS}. The Almgren monotonicity formula (see \cite{zbMATH03686256, zbMATH01528183}) is of fundamental importance in the study of harmonic functions. It is also a key ingredient in the study of unique continuation, and is used extensively in free boundary
problems (see, for example, \cite{zbMATH06062397}). In the context of the infinite $d-$regular tree, we generalize the Almgren type monotonicity formula of \cite{SVGS} for $p\ge 1$ in Theorem \ref{T:Almgren}.

In \S \ref{S:ex} we provide examples of harmonic functions in the infinite $2-$ and $3-$regular trees, and compute their weighted Dirichlet energy, Weiss, and Almgren frequency functions.

\section{Acknowledgments}
This material is based upon work supported by the National Science Foundation under Award No. DMS-2348739. M.S.V.G. was also partially supported by a Karen EDGE fellowship.

\section{Monotonicity of a weighted Dirichlet energy}\label{S:Dirichlet}

Our first main result, Theorem \ref{T:Dirichlet} below, can be thought of as a generalization of the fact that for ``standard'' harmonic functions in $\R^n$, the function
$r\mapsto \frac{1}{r^n}\int_{B_r(x_0)}|\nabla u|^2$ 
is monotonic non-decreasing. A first natural related question in the discrete setting is:

\begin{question}\label{Q:Dirichlet} Can one define a function $f(k)$ with $\lim\limits_{k\rightarrow\infty}f(k)\rightarrow \infty$ so that if \[
D_k=\sum_{\substack{(y,z)\in E\\ y\in V_k, \ z\in V_{k+1}}}|u(y)-u(z)|^2,
\]then 
\[
F(k)=\frac{1}{f(k)}\sum_{l=0}^{k-1}D_{\ell}
\]
is monotonic non-decreasing whenever $u$ is harmonic?
\end{question}

In $\R^n$, note that if $u$ is harmonic and not constant,  $\int_{B_r(x_0)}|\nabla u|^2 \rightarrow \infty$ as $r\rightarrow \infty$. The same is not always true for harmonic functions in graphs, as was observed in Example 6.2.2 of \cite{Davide} for a non-tree graph. We now give an example for the $3-$regular tree:

\begin{ex}\label{needweight}
    In the $3-$regular tree with root $x_0$, define $u(x_0)=0$. If the children of $x_0$ are $y_1,y_2,y_3,$ let $u(y_1)=1$, $u(y_2)=u(y_3)=\frac{-1}{2}$. Assume $u$ has been defined up to $V_k$, with $k\ge 1$. Let $v\in V_{k+1}$ and call $v$'s parent $v_p$, and $v$'s children $z_1$ and $z_2$. Define $u(z_1)=u(z_2)=\frac{3u(v)-u(v_p)}{2}$ (see Figure \ref{f:example}). 
Then $u$ is harmonic. If $\ell\ge 1$, $w\in V_{\ell}$  has children $w_1,w_2$ and parent $w_p$, then
\[
u(w)-u(w_i)=u(w)-\frac{3u(w)-u(w_p)}{2}=\frac{u(w_p)-u(w)}{2}.
\]Therefore
\[
D_{\ell}=\sum_{w\in V_{\ell}}2\frac{1}{4}(u(w)-u(w_p))^2=\frac{1}{2}D_{k}.
\]
Using $D_0=\frac{3}{2}$ we conclude hence $D_k=\frac{3}{2\cdot 2^{k}}$. Therefore
\[
\sum_{j=0}^{\infty}D_j=\sum_{j=0}^{\infty}\frac{3}{2\cdot2^{j}}=3.
\]

\begin{center}
\begin{figure}[htbp]
\begin{tikzpicture}[
    scale=1,
    vdot/.style={circle, fill=black, inner sep=1.1pt},
    vlabel/.style={font=\tiny, inner sep=2pt}, 
    edge/.style={draw=black!60, thin}
]

    \node[vdot] (x0) at (0,0) {};
    \node[vlabel, anchor=south east] at (x0) {$0$};

    \node[vdot] (c1) at (90:0.8) {};
    \node[vlabel, anchor=east] at (c1) {$1$};
    \draw[edge] (x0) -- (c1);
    
    \node[vdot] (c2) at (210:0.8) {};
    \node[vlabel, anchor=north east] at (c2) {$-\frac{1}{2}$};
    \draw[edge] (x0) -- (c2);
    
    \node[vdot] (c3) at (330:0.8) {};
    \node[vlabel, anchor=north west] at (c3) {$-\frac{1}{2}$};
    \draw[edge] (x0) -- (c3);

    \node[vdot] (c11) at (75:1.5) {};
    \node[vlabel, anchor=east] at (c11) {$\frac{3}{2}$};
    \node[vdot] (c12) at (105:1.5) {};
    \node[vlabel, anchor=west] at (c12) {$\frac{3}{2}$};
    \draw[edge] (c1) -- (c11) (c1) -- (c12);

    \node[vdot] (c21) at (195:1.5) {};
    \node[vlabel, anchor=south east] at (c21) {$-\frac{3}{4}$};
    \node[vdot] (c22) at (225:1.5) {};
    \node[vlabel, anchor=north] at (c22) {$-\frac{3}{4}$};
    \draw[edge] (c2) -- (c21) (c2) -- (c22);

    \node[vdot] (c31) at (315:1.5) {};
    \node[vlabel, anchor=north] at (c31) {$-\frac{3}{4}$};
    \node[vdot] (c32) at (345:1.5) {};
    \node[vlabel, anchor=south west] at (c32) {$-\frac{3}{4}$};
    \draw[edge] (c3) -- (c31) (c3) -- (c32);

    \node[vdot] (c111) at (68:2.1) {}; \node[vlabel, anchor=south] at (c111) {$\frac{7}{4}$};
    \node[vdot] (c112) at (82:2.1) {}; \node[vlabel, anchor=south] at (c112) {$\frac{7}{4}$};
    \draw[edge] (c11) -- (c111) (c11) -- (c112);

    \node[vdot] (c121) at (98:2.1) {}; \node[vlabel, anchor=south] at (c121) {$\frac{7}{4}$};
    \node[vdot] (c122) at (112:2.1) {}; \node[vlabel, anchor=south] at (c122) {$\frac{7}{4}$};
    \draw[edge] (c12) -- (c121) (c12) -- (c122);

    \node[vdot] (c211) at (188:2.1) {}; \node[vlabel, anchor=east] at (c211) {$-\frac{7}{8}$};
    \node[vdot] (c212) at (202:2.1) {}; \node[vlabel, anchor=east] at (c212) {$-\frac{7}{8}$};
    \draw[edge] (c21) -- (c211) (c21) -- (c212);

    \node[vdot] (c221) at (218:2.1) {}; \node[vlabel, anchor=north east] at (c221) {$-\frac{7}{8}$};
    \node[vdot] (c222) at (232:2.1) {}; \node[vlabel, anchor=north] at (c222) {$-\frac{7}{8}$};
    \draw[edge] (c22) -- (c221) (c22) -- (c222);

    \node[vdot] (c311) at (308:2.1) {}; \node[vlabel, anchor=north] at (c311) {$-\frac{7}{8}$};
    \node[vdot] (c312) at (322:2.1) {}; \node[vlabel, anchor=north west] at (c312) {$-\frac{7}{8}$};
    \draw[edge] (c31) -- (c311) (c31) -- (c312);

    \node[vdot] (c321) at (338:2.1) {}; \node[vlabel, anchor=west] at (c321) {$-\frac{7}{8}$};
    \node[vdot] (c322) at (352:2.1) {}; \node[vlabel, anchor=west] at (c322) {$-\frac{7}{8}$};
    \draw[edge] (c32) -- (c321) (c32) -- (c322);
\end{tikzpicture}
\caption{Harmonic function on the infinite 3-regular tree described in Example \ref{needweight} (locally).}
\label{f:example}
\end{figure}
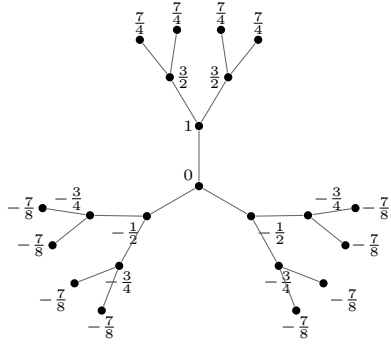
\end{center}
\end{ex}

This leads to the negative answer to Question \ref{Q:Dirichlet}, and indicates that weights might be needed to prove a discrete version of the monotonicity of $\frac{1}{r^n}\int_{B_r}|\nabla u|^2$.  In this paper, we prove the following monotonicity formula for a weighted Dirichlet energy:
\begin{theorem}[Discrete weighted Dirichlet energy]\label{T:Dirichlet}
    Let $x_0$ be a vertex in the infinite $d-$regular tree $T_d=(V,E)$ and let $u$ be a harmonic function on $T_d$. Then, for any $1 \leq p <\infty$
    \[
    G(k) = \frac{1}{k} \sum_{\substack{e \in E\\ d(x_0,e) < k}} ((d-1)^{p-1})^{d(x_0,e)} \cdot |\nabla u(e)|^p
    \]
    is monotonically non-decreasing in $k$.
\end{theorem}

Example \ref{E:bounded}
shows that the exponential factor present is in general necessary.

 \begin{proof}
Let $k\ge 1$. We first explicitly write
\begin{align*}
  \sum_{\substack{e \in E\\d(x_0,e) < k}} |\nabla u(e)|^p= \sum_{\ell=0}^{k-1} \sum_{\substack{(a,b) \in E \\a\in V_{\ell}, \ b\in V_{\ell+1}}} |u(a) - u(b)|^p.
\end{align*}

\begin{center}
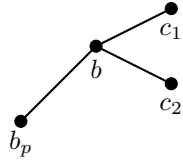
\begin{figure}[h!]
\begin{tikzpicture}
\filldraw (0,0) circle (0.08cm);
\filldraw (1,1) circle (0.08cm);
\filldraw (2,1.5) circle (0.08cm);
\filldraw (2,0.5) circle (0.08cm);
\draw [thick] (0,0) -- (1,1);
\draw [thick] (2,1.5) -- (1,1);
\draw [thick] (2,0.5) -- (1,1);
\node at (0, -0.3) {$b_p$};
\node at (1, 1-0.3) {$b$};
\node at (2, 1.5-0.3) {$c_{1}$};
\node at (2, 0.5-0.3) {$c_{2}$};
\end{tikzpicture}
\caption{A sketch for the argument when $d=3$.}
\end{figure}
\end{center}

We consider an individual branch in the tree: consider a vertex $b$, its ancestor (or ``parent'') $b_p$ and its descendants $c_1, \dots, c_{d-1}$. Then
\[ 
u(b_p) + \sum_{i=1}^{d-1} u(c_i) = d \cdot u(b)\]
which can be rewritten as
\[ u(b) - u(b_p) = \sum_{i=1}^{d-1} (u(c_i) - u(b))\]
and thus
\begin{equation}\label{ACF1} |u(b) - u(b_p)| \leq \sum_{i=1}^{d-1} |u(c_i) - u(b)|.
\end{equation}
Using H\"older's inequality, we get
\begin{equation}\label{ACF2}
\sum_{i=1}^{d-1} |u(c_i) - u(b)| \leq \left(  \sum_{i=1}^{d-1} |u(c_i) - u(b)|^p\right)^{1/p} (d-1)^{1- \frac{1}{p}}.
\end{equation}
Combining \eqref{ACF1} and \eqref{ACF2}, we deduce
\begin{equation}\label{ACF3}
    \sum_{i=1}^{d-1} |u(c_i) - u(b)|^p \geq \frac{|u(b) - u(b_p)|^{p}}{(d-1)^{p-1}}.
\end{equation} 

Now fix $\ell\ge 0$. Summing \eqref{ACF3} over $b\in V_{\ell+1}$ we obtain
\begin{align*}
\sum_{\substack{(z,w)\in E\\z\in V_{\ell+1}, \ w\in V_{\ell+2}}}|u(w)-u(z)|^p&=\sum_{b\in V_{\ell+1}}\sum_{i=1}^{d-1} |u(c_i) - u(b)|^p\\
&\geq \sum_{b\in V_{\ell}+1}\frac{|u(b) - u(b_p)|^{p}}{(d-1)^{p-1}}=\sum_{\substack{(c,d)\in E\\c\in V_{\ell}, \ d\in V_{\ell+1}}}\frac{|u(c) - u(d)|^{p}}{(d-1)^{p-1}}.
\end{align*}
This, in turn, implies that the function
\[
F(\ell)=(d-1)^{(p-1)\ell} \sum_{\substack{(x,y) \in E\\x\in V_{\ell}, \ y\in V_{\ell+1}}} |u(x) - u(y)|^p\]
is monotonically non-decreasing in $\ell$, as 
\begin{align*}
F(\ell+1)&=(d-1)^{(p-1)(\ell+1)}\sum_{\substack{(z,w) \in E\\ z\in V_{\ell+1}, \ w\in V_{\ell+2}}} |u(z) - u(w)|^p\\
&\ge (d-1)^{(p-1)(\ell+1)} \sum_{\substack{(c,d) \in E\\ c\in V_{\ell}, \ d\in V_{\ell+1}}} \frac{|u(c) - u(d)|^p}{(d-1)^{p-1}}=F(\ell).
\end{align*}
Now, recall
\begin{align*}
G(k)&=\frac{1}{k}\sum_{\substack{e \in E\\d(x_0,e) < k}} ((d-1)^{p-1})^{d(e,x_0)}|\nabla u(e)|^p= \frac{1}{k}\sum_{\ell=0}^{k-1} \sum_{\substack{(a,b) \in E\\a\in V_{\ell}, \ b\in V_{\ell+1}}} (d-1)^{(p-1)\ell}|u(a) - u(b)|^p\\
&=\frac{1}{k}\sum_{\ell=0}^{k-1}(d-1)^{(p-1)\ell}\sum_{\substack{(a,b) \in E\\a\in V_{\ell}, \ b\in V_{\ell+1}}}|u(a) - u(b)|^p=\frac{1}{k}\sum_{\ell=0}^{k-1}F(\ell).
\end{align*}
Therefore, it suffices to prove that for all $k\ge 1$, 
\[
\frac{1}{k}\sum_{\ell=0}^{k-1}F(\ell) \le \frac{1}{k+1}\sum_{\ell=0}^{k}F(\ell).
\]
This, in turn, is equivalent to 
\[
\left(1+\frac{1}{k}\right)\sum_{\ell=0}^{k-1}F(\ell)=\frac{k+1}{k}\sum_{\ell=0}^{k-1}F(\ell) \le\sum_{\ell=0}^{k}F(\ell).
\]
Consequently, it suffices to prove
\[
\frac{1}{k}\sum_{\ell=0}^{k-1}F(\ell)\le F(k).\] 
Since $F$ is monotonic non-decreasing,  $F(0),F(1),\ldots,F(k-1)\le F(k)$, hence the result holds.

 \end{proof}

\section{Weiss monotonicity formula}\label{S:Weiss}

For harmonic functions in $\R^n$, the functional
\[
W(r)=\frac{1}{r^{n+2}}\int_{B_r}|\nabla u|^2-\frac{2}{r^{n+3}}\int_{S_r}u^2
\]
is monotonic. Weiss proved in \cite{Wei98} and \cite{Wei99}  analogous results in the context of the classical obstacle problem. 
The main results of this section, Theorems \ref{T:Weiss} and \ref{T:Weiss2d} are inspired by this result:

\subsection{Weiss monotonicity formula for the infinite $d-$regular tree with $d\ge 3$}

\begin{theorem}\label{T:Weiss}For any harmonic function $u$ in the infinite $d-$regular tree with root $x_0$, the quantity
\[
W(k) = \sum_{\substack{e\in E\\ d(x_0,e)<k}}(d-1)^{d(x_0,e)}|\nabla u(e)|^2- \frac{d-2}{2}\frac{1}{(d-1)^{k-1}}\sum_{y\in V_k}u(y)^2
\]
is monotone non-decreasing for $k\ge 1$.
\end{theorem}

\begin{proof}
We first rewrite $W(k)$. Notice that 
\[
\sum_{\substack{e\in E\\ d(x_0,e)<k}}(d-1)^{d(x_0,e)}|\nabla u(e)|^2=\sum_{j=0}^{k-1} (d-1)^j D_j,
\]
where if $z_{\ell}$ denotes the children of $y\in V_j$, and $y_{\ell}$ denotes the children of the root $x_0$, then for $j\ge 1$
\[
D_j=\sum_{y\in V_j}\sum_{\ell=1}^{d-1}|u(y)-u(z_{\ell})|^2, \qquad D_0=\sum_{\ell=1}^d|u(x_0)-u(y_{\ell})|^2.
\]
Call $H_k=\sum_{y\in V_k}u(y)^2$, so that
\[
W(k)=\sum_{j=0}^{k-1} (d-1)^j D_j-\frac{d-2}{2}\frac{H_k}{(d-1)^{k-1}}.
\]
We have
\begin{align*}
W(k+1)-W(k)&=(d-1)^kD_k+\left(\frac{d-2}{2}\right)\frac{(d-1)H_k-H_{k+1}}{(d-1)^{k}}\\
&=(d-1)^kD_k-\left(\frac{d-2}{2}\right)\frac{N_k}{(d-1)^{k}},
\end{align*}
where 
\begin{equation}\label{X}
N_k := H_{k+1}-(d-1)H_k.
\end{equation}
Thus
\begin{align*}
W(k+1)&\ge W(k) \iff (d-1)^kD_k\ge \left(\frac{d-2}{2}\right)\frac{N_k}{(d-1)^{k}} \\
&\iff (d-1)^{2k}D_k\ge \left(\frac{d-2}{2}\right)N_k.
\end{align*}

Notice that in addition to a monotonicity result, the main theorem of \cite{SVGS} also proves that $N_k\ge 0$. The key ingredient in our proof will be the proof that $N_k\le C_d(d-1)^kD_k$ for a specific constant $C_d.$ This will lead to $(d-1)^{2k}D_k\ge\frac{d-2}{2}N_k$ by the choice of $C_d$.
From harmonicity, if $v_p$ denotes the parent of $v$ and the $d-1$ children of $v$ are denoted by $c_1^v,\ldots,c_{d-1}^v$, 
\[
du(v)=u(v_p)+\sum_{j=1}^{d-1}u(c_j^v).
\]
Therefore,
\begin{equation}\label{c1c2v}
\sum_{j=1}^{d-1}u(c_j^v)-(d-1)u(v)=u(v)-u(v_p).
\end{equation}

We notice that for any constants $A_1,\ldots,A_{d-1}$,
\begin{equation}\label{ineq}
\sum_{j=1}^{d-1}A_j^2=\frac{1}{d-1}\left(\sum_{j=1}^{d-1}A_j\right)^2+\frac{1}{d-1}\sum_{1\le i<j\le d-1}(A_i-A_j)^2.\end{equation}
Let $v\in V_k$. Applying \eqref{ineq} with $A_j=u(v)-u(c_j^v)$ and using \eqref{c1c2v}, we obtain 
\begin{equation}\label{sumofdiff}
\sum_{j=1}^{d-1}(u(v)-u(c_j^v))^2
= \frac{1}{d-1}\left(u(v)-u(v_p)\right)^2+\frac{1}{d-1}\sum_{1\le i<j\le d-1}(u(c_i^v)-u(c_j^v))^2.
\end{equation}

Let 
\[
R_k:=\sum_{\substack{v\in V_k\\1\le i<j\le d-1}}(u(c_i^v)-u(c_j^v))^2\geq 0.\] Summing \eqref{sumofdiff} over $v\in V_k$ gives
\begin{equation}\label{Ak}
D_k=\frac{1}{d-1}D_{k-1}+\frac{1}{d-1}R_k.
\end{equation}
We can rewrite \eqref{Ak} as 
\begin{equation}\label{Ak2}
(d-1)^kD_k-(d-1)^{k-1}D_{k-1}=(d-1)^{k-1}R_k\ge0,
\end{equation}
which implies $(d-1)^kD_k$ is non-decreasing.

Expanding the square,
\[
D_k=\sum_{v\in V_k}\sum_{i=1}^{d-1}(u(v)-u(c_i^v))^2=(d-1)\sum_{v\in V_k}u(v)^2-2\sum_{v\in V_k}\left(u(v)\left(\sum_{j=1}^{d-1}u(c_j^v)\right)\right)+\sum_{\substack{v\in V_k\\1\le i\le d-1}}u(c_i^v)^2.
\]

By harmonicity, $\sum_{j=1}^{d-1}u(c_j^v)=du(v)-u(v_p)$ hence
\begin{equation}\label{Dnew}
D_k=(d-1)\sum_{v\in V_k}u(v)^2-2\sum_{v\in V_k}u(v)(du(v)-u(v_p))+\sum_{\substack{v\in V_k\\1\le i\le d-1}}u(c_i^v)^2.
\end{equation}
That is, if we define for $k\ge 1$ 
\[
C_k=\sum_{v\in V_k}u(v)u(v_p)=\sum_{z\in V_{k-1}}\sum_{j=1}^{d-1}u(z)u(c_j^z),
\]
then \eqref{Dnew} becomes
\begin{equation}\label{newDk}
D_k=\sum_{\substack{v\in V_k\\1\le i\le d-1}}u(c_i^v)^2-(d+1)\sum_{v\in V_k}u(v)^2+2\sum_{v\in V_k}u(v)u(v_p)=H_{k+1}-(d+1)H_k+2C_k.
\end{equation}
By definition, $C_{k-1}=\sum_{w\in V_{k-1}}u(w)u(w_p)$. By harmonicity, $u(w_p)=du(w)-\sum_{j=1}^{d-1}u(c_j^w)$, hence 
\[
C_{k-1}=\sum_{w\in V_{k-1}}\left(-\sum_{j=1}^{d-1}u(w)u(c_j^w)+du(w)^2\right)=dH_{k-1}-C_k.\]

Therefore
\begin{equation}\label{Ck}
C_{k}=dH_{k-1}-C_{k-1}.
\end{equation}
Combining \eqref{newDk} and \eqref{Ck} we obtain
\[
D_k+D_{k-1}
=H_{k+1}-dH_k+(d-1)H_{k-1}.
\]

Using the definition of $N_k$ from \eqref{X}, we then have
\begin{equation}\label{XkA}
N_k=H_{k+1}-(d-1)H_k = (H_{k+1}-dH_k+(d-1)H_{k-1})
+(H_k-(d-1)H_{k-1})= D_k+D_{k-1}+N_{k-1}.
\end{equation}

From \eqref{Ak2}, we have for any $0\le j\le k$,
\begin{equation}\label{doubling2}
D_j\le \frac{(d-1)^kD_k}{(d-1)^j}.
\end{equation}
In particular:
\begin{equation}\label{Ak-1Ak}
    D_{k-1}\le (d-1)D_k.
\end{equation}

We claim that $N_1=D_1+2D_0.$
Indeed, since $C_1=dH_0$, \eqref{newDk} gives
\[
D_1=H_2-(d+1)H_1+2dH_0.
\]

Then
\[
N_1=H_2-(d-1)H_1
= D_1+2H_1-2dH_0.
\]

But we have, simply writing
\[
D_0=\sum_{j=1}^{d}(u(x_0)-u(y_j))^2=du(x_0)^2-2u(x_0)\sum_{j=1}^du(y_j)+\sum_{j=1}^du(y_j)^2=H_1-dH_0,
\]
so $2D_0=2H_1-2dH_0.$

Thus:
\begin{equation}\label{X1}
N_1=D_1+2D_0.
\end{equation}
Let $C_d=\frac{2d-1}{d-1}+\frac{d}{(d-1)(d-2)}=2+\frac{2}{d-2}$. We prove that for $k\ge 1$
\begin{equation}\label{XkAk}
N_k\le C_d(d-1)^kD_k.
\end{equation}
First, let us deal with the case $k=1$. From \eqref{X1} and \eqref{Ak-1Ak},
\[
N_1 = D_1 + 2D_0 \le D_1 + 2(d-1)D_1 = (2d-1)D_1 \le C_d(d-1)D_1.
\]

Now assume $k\ge 2$. Equation \eqref{XkA} gives
\begin{equation}\label{moreAk}
    N_k=D_1+2D_0+\sum_{j=2}^k(D_j+D_{j-1}).
\end{equation}

We have by \eqref{doubling2}
\[
D_1\le\frac{(d-1)^kD_k}{d-1}, \qquad D_0\le \frac{(d-1)^kD_k}{1}.
\]
Therefore
\begin{equation}\label{this}
D_1+2D_0\le \frac{(d-1)^kD_k}{d-1}+2(d-1)^kD_k=(d-1)^kD_k\left(\frac{1}{d-1}+2\right)=\frac{2d-1}{d-1}(d-1)^kD_k.
\end{equation}
Inequality \eqref{doubling2} also gives
\begin{equation}\label{thistoo}
\begin{aligned}
\sum_{j=2}^k(D_j+D_{j-1})&\le (d-1)^kD_k\sum_{j=2}^k\left(\frac{1}{(d-1)^j}+\frac{1}{(d-1)^{j-1}}\right)=(d-1)^kD_k\sum_{j=2}^k\frac{d}{(d-1)^j}\\
&<\frac{d}{(d-1)(d-2)}(d-1)^kD_k.
\end{aligned}
\end{equation}

Adding \eqref{this} and \eqref{thistoo}, \eqref{moreAk} gives for $k\ge 2$:
\begin{equation}\label{almostthere}
N_k\le \frac{2d-1}{d-1}(d-1)^kD_k+\frac{d}{(d-1)(d-2)}(d-1)^kD_k=C_d(d-1)^kD_k.
\end{equation}

Recall that we want to prove that $(d-1)^{2k}D_k\ge \frac{d-2}{2}N_k$. Given $k\ge 2$, \eqref{almostthere} leads to
\[
\frac{d-2}{2}N_k\le \frac{d-2}{2}\left(2+\frac{2}{d-2}\right)(d-1)^{k} D_k\le (d-1)^{2k}D_k,
\]
since $d\ge 3$ and $k\ge 1$.

\end{proof}

\subsection{Weiss for the infinite 2-regular tree}

We work on the infinite 2-regular tree rooted at $0$. We have the following

\begin{lemma}\label{L:increments} If $u$ is harmonic, for all $j\ge 0$, $u(j)-u(j+1)=u(0)-u(1)$.
\end{lemma}

\begin{proof}
The equality is trivial for $j=0$. Assume $u(j)-u(j+1)=u(0)-u(1)$. Since $u$ is harmonic, $2u(j+1)=u(j)+u(j+2)$, hence $u(j+1)-u(j+2)=u(j)-u(j+1)$. Using the induction hypothesis, we conclude that $u(j+1)-u(j+2)=u(0)-u(1)$, as desired.
\end{proof}

\begin{lemma}\label{L:harmonic2d} If $u$ is harmonic, for all $j\ge 0$, $u(j)=ju(1)-(j-1)u(0)$. In particular,
\[
u(j)=j(u(1)-u(0))+u(0).
\]
\end{lemma}

\begin{proof}
The equality is trivial for $j=0$. Assume that $u(j)=ju(1)-(j-1)u(0)$ for some $j\ge 0$. By harmonicity,
$2u(j)=u(j-1)+u(j+1)$, hence $u(j)-u(j+1)=u(j-1)-u(j)=u(0)-u(1)$, by Lemma \ref{L:increments}. Combining this with the induction hypothesis, we obtain
\[
u(j+1)=u(j)-u(0)+u(1)=(j+1)u(1)-ju(0)=(j+1)(u(1)-u(0))+u(0).
\]
\end{proof}
Notice that Lemma \ref{L:harmonic2d} shows that harmonic functions on the infinite $2-$regular tree are completely determined by the values $u(0)$ and $u(1)$. 

\begin{theorem}\label{T:Weiss2d}
   Let $u$ be harmonic in the infinite $2$-regular tree. Then
\[
W(k)=\frac{1}{k}\sum_{\substack{e\in E\\ d(0,e)< k}}|\nabla u(e)|^2-\frac{1}{2k^2}\sum_{v\in V_k}u^2(v).
\]
is monotonic in $k$ (with $k\ge 1$) and $W(k)\rightarrow (u(0)-u(1))^2$ as $k\rightarrow\infty$.
\end{theorem}

\begin{proof}
By Lemma \ref{L:harmonic2d}, $u(n)=an+b$, where $a=u(1)-u(0)$ and $b=u(0).$
We have
\[
|u(j+1)-u(j)|=|a(j+1)+b-(aj+b)|=|a|,
\]
therefore
\begin{align*}
D(k)&:=\sum_{\substack{e\in E\\ d(0,e)< k}}|\nabla u(e)|^2=\sum_{j=0}^{k-1}|u(j)-u(j+1)|^2+\sum_{j=0}^{k-1}|u(-j)-u(-j-1)|^2\\
&=2\sum_{j=0}^{k-1}a^2=2a^2k.
\end{align*}
Moreover, $\sum_{v\in V_k}u^2(v)=u^2(k)+u^2(-k)=2a^2k^2+2b^2,$ hence
\[
W(k)=\frac{1}{k}2a^2k-\frac{1}{2k^2}(2a^2k^2+2b^2)=a^2-\frac{b^2}{k^2}\rightarrow a^2 \text{ as } k\rightarrow\infty.
\]
\end{proof}

\section{Monotonicity of an Almgren frequency function with $p\ge 1$}
In this section, we revisit the Almgren monotonicity formula proved in \cite{SVGS}. In the context of $d-$regular trees, the formula from \cite{SVGS} can be rewritten as the following:
\begin{thm}[Theorem in 1.2 from \cite{SVGS}] Let $u$ be harmonic in the infinite $d-$regular tree. Then for $k\ge 0$,
\[
N(k)=\sum_{v\in V_{k+1}}u(v)^2-(d-1)\sum_{v\in V_k}u(v)^2
\]
is monotone non-decreasing.    
\end{thm}

In this paper, we generalize this result to any $p\ge $1 below:

\begin{theorem}[Discrete Almgren]\label{T:Almgren} Fix $p\ge 1$ and $d\ge 2$. Let $x$ be a vertex in the infinite $d-$regular tree $T_d=(V,E)$, and let $u:V \rightarrow \mathbb{R}$ be a harmonic function. Then, for $k \geq 0$, the function
\[
N(k) = \frac{1}{d-1} \sum_{d(x,y)=k+1} |u(y)|^p - \sum_{d(x,y)=k}|u(y)|^p 
\]
is monotonically increasing in $k$, i.e. $N(k+1) \geq N(k).$
\end{theorem}

\begin{proof}
Fix $p\ge 1$ and $d\ge 2$.
    Define 
    \[
    H_k=\sum_{x\in V_k}|u(x)|^p.
    \]
    The monotonicity of $N$ is equivalent to $N(k+1)\ge N(k)$ for $k\ge 0$, which in turn is equivalent to
    \[
    \frac{1}{d-1}H_{k+2}-H_{k+1}\ge \frac{1}{d-1}H_{k+1}-H_k.
    \]
    Rearranging the terms, this is equivalent to
    \begin{equation}\label{want}
    \frac{1}{d-1}H_{k+2}-\frac{d}{d-1}H_{k+1}+H_k\ge0.
    \end{equation}

    Fix $w\in V_{k+1}$. Denote the $d-1$ children of $w$ by $w_1,\ldots,w_{d-1}\in V_{k+2}.$
    By H\"{o}lder's inequality,
    \[
  \sum_{j=1}^{d-1}|u(w_j)|= \sum_{j=1}^{d-1}1\cdot |u(w_j)|\le(d-1)^{1-\frac{1}{p}}\left(\sum_{j=1}^{d-1}|u(w_j)|^p\right)^{\frac{1}{p}},
    \]
    therefore
    \begin{equation}\label{Holder}
    \left(\sum_{j=1}^{d-1}|u(w_j)|\right)^p\le(d-1)^{p-1}\sum_{j=1}^{d-1}|u(w_j)|^p.
    \end{equation}
    We can rewrite \eqref{Holder} as
    \begin{equation}\label{holder2}
    \frac{1}{d-1}\sum_{j=1}^{d-1}|u(w_j)|^p\ge\left(\frac{1}{d-1}\sum_{j=1}^{d-1}|u(w_j)|\right)^p.
    \end{equation}
    Since $u$ is harmonic, if we denote $w$'s parent by $w_p$, we have
\[
u(w)=\frac{u(w_p)+\sum_{j=1}^{d-1}u(w_j)}{d},
\]
hence
\begin{equation}\label{harwj}
\sum_{j=1}^{d-1}u(w_j)=du(w)-u(w_p).
\end{equation}
Combining \eqref{holder2} and \eqref{harwj} we obtain
\begin{equation}\label{good}
\sum_{j=1}^{d-1}|u(w_j)|^p\ge(d-1)\left|\frac{du(w)-u(w_p)}{d-1}\right|^p.
\end{equation}
Summing \eqref{good} over $w\in V_{k+1}$, we get
\begin{equation}\label{ha!}
H_{k+2}\ge(d-1)\sum_{w\in V_{k+1}}\left|\frac{du(w)-u(w_p)}{d-1}\right|^p.
\end{equation}
To address the right-hand-side, we will use the following inequality, which holds for any $A,B\in \R$, $d\ge 2$ and $p\ge 1$:
\begin{equation}\label{Jensen}
\left|\frac{dA-B}{d-1}\right|\ge\frac{d}{d-1}|A|^p-\frac{1}{d-1}|B|^p.
\end{equation}
To prove it, write
\[
A=\frac{B}{d}+\frac{d-1}{d}\frac{(dA-B)}{d-1}.
\]
Since $f(x)=|x|^p$ is convex for $p\ge 1$, Jensen's inequality gives
\[
|A|^p\le \frac{|B|^p}{d}+\frac{d-1}{d}\left|\frac{dA-B}{d-1}\right|^p,
\]
from which \eqref{Jensen} follows. We now apply \eqref{Jensen} with $A=u(w)$ and $B=u(w_p)$:
\[
\left|\frac{du(w)-u(w_p)}{d-1}\right|^{p}\ge\frac{d}{d-1}|u(w)|^p-\frac{1}{d-1}|u(w_p)|^p.
\]
Summing over $w\in V_{k+1}$, we get
\begin{equation}\label{nice!}
\sum_{w\in V_{k+1}}\left|\frac{du(w)-u(w_p)}{d-1}\right|^{p}\ge\frac{d}{d-1}H_{k+1}-\frac{1}{d-1}(d-1)H_k=\frac{d}{d-1}H_{k+1}-H_k.
\end{equation}

Combining \eqref{ha!} with \eqref{nice!} we conclude
\[
H_{k+2}\ge dH_{k+1}-(d-1)H_k,
\]
from which we conclude
\[
\frac{1}{d-1}H_{k+2}-\frac{d}{d-1}H_{k+1}+H_k\ge 0,
\]
which is precisely what we wanted to prove.
\end{proof}

\section{Examples}\label{S:ex}
The purpose of this section is to construct explicit examples of harmonic functions to illustrate Theorems \ref{T:Dirichlet}, \ref{T:Weiss}, \ref{T:Weiss2d} and \ref{T:Almgren}.

\subsection{For the infinite $3$-regular tree}
\subsubsection{Bounded Harmonic Function}\label{E:bounded}
\begin{center}
\begin{figure}[htbp]
\begin{tikzpicture}[
    scale=1.4,
    vdot/.style={circle, fill=black, inner sep=1.2pt},
    vlabel/.style={font=\tiny, inner sep=3pt},
    edge/.style={draw=black!60, thin}
]

    \node[vdot] (x0) at (0,0) {};
    \node[vlabel, anchor=south east] at (x0) {$0$};

    \node[vdot] (n1) at (90:0.8) {}; \node[vlabel, anchor=south] at (n1) {$1$};
    \node[vdot] (n2) at (210:0.8) {}; \node[vlabel, anchor=north east] at (n2) {$-1$};
    \node[vdot] (n3) at (330:0.8) {}; \node[vlabel, anchor=west] at (n3) {$0$};
    
    \draw[edge] (x0) -- (n1);
    \draw[edge] (x0) -- (n2);
    \draw[edge] (x0) -- (n3);

    \node[vdot] (n11) at (75:1.5) {}; \node[vlabel, anchor=south west] at (n11) {$\frac{3}{2}$};
    \node[vdot] (n12) at (105:1.5) {}; \node[vlabel, anchor=south east] at (n12) {$\frac{3}{2}$};
    \draw[edge] (n1) -- (n11); \draw[edge] (n1) -- (n12);

    \node[vdot] (n21) at (195:1.5) {}; \node[vlabel, anchor=east] at (n21) {$-\frac{3}{2}$};
    \node[vdot] (n22) at (225:1.5) {}; \node[vlabel, anchor=north east] at (n22) {$-\frac{3}{2}$};
    \draw[edge] (n2) -- (n21); \draw[edge] (n2) -- (n22);

    \node[vdot] (n31) at (315:1.5) {}; \node[vlabel, anchor=north west] at (n31) {$0$};
    \node[vdot] (n32) at (345:1.5) {}; \node[vlabel, anchor=west] at (n32) {$0$};
    \draw[edge] (n3) -- (n31); \draw[edge] (n3) -- (n32);

    \node[vdot] (n111) at (68:2.1) {}; \node[vlabel, anchor=south west] at (n111) {$\frac{7}{4}$};
    \node[vdot] (n112) at (82:2.1) {}; \node[vlabel, anchor=south] at (n112) {$\frac{7}{4}$};
    \node[vdot] (n121) at (98:2.1) {}; \node[vlabel, anchor=south] at (n121) {$\frac{7}{4}$};
    \node[vdot] (n122) at (112:2.1) {}; \node[vlabel, anchor=south east] at (n122) {$\frac{7}{4}$};
    \draw[edge] (n11) -- (n111); \draw[edge] (n11) -- (n112);
    \draw[edge] (n12) -- (n121); \draw[edge] (n12) -- (n122);

    \node[vdot] (n211) at (188:2.1) {}; \node[vlabel, anchor=east] at (n211) {$-\frac{7}{4}$};
    \node[vdot] (n212) at (202:2.1) {}; \node[vlabel, anchor=east] at (n212) {$-\frac{7}{4}$};
    \node[vdot] (n221) at (218:2.1) {}; \node[vlabel, anchor=north east] at (n221) {$-\frac{7}{4}$};
    \node[vdot] (n222) at (232:2.1) {}; \node[vlabel, anchor=north] at (n222) {$-\frac{7}{4}$};
    \draw[edge] (n21) -- (n211); \draw[edge] (n21) -- (n212);
    \draw[edge] (n22) -- (n221); \draw[edge] (n22) -- (n222);

    \node[vdot] (n311) at (308:2.1) {}; \node[vlabel, anchor=north] at (n311) {$0$};
    \node[vdot] (n312) at (322:2.1) {}; \node[vlabel, anchor=north west] at (n312) {$0$};
    \node[vdot] (n321) at (338:2.1) {}; \node[vlabel, anchor=west] at (n321) {$0$};
    \node[vdot] (n322) at (352:2.1) {}; \node[vlabel, anchor=west] at (n322) {$0$};
    \draw[edge] (n31) -- (n311); \draw[edge] (n31) -- (n312);
    \draw[edge] (n32) -- (n321); \draw[edge] (n32) -- (n322);
\end{tikzpicture}
\caption{Sketch of the bounded harmonic function described in \S \ref{E:bounded} (locally).}
 \label{fig:4}
\end{figure}
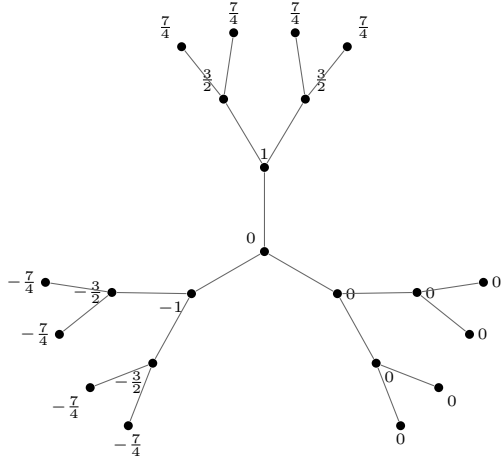
\end{center}

The first example comes from \cite{SVGS}, and is sketched in Fig. \ref{fig:4}. We fix a root $x_0$ for the infinite $3-$regular tree and set $u(x_0)=0$. The function $u$ is also zero in an entire branch leading away from it. Define
\[ a_0 = 0, \  a_1 = 1 \qquad \mbox{and} \qquad a_{k+1} = \frac{3a_k - a_{k-1}}{2}.\]
We can prove by induction that for $n \geq 1$
$$ a_n = 2 - \frac{1}{2^{n-1}}.$$
On the remaining branch, the values are simply $-a_n$. This leads to a non-constant harmonic function with values between $-2$ and $2$.

Our discrete weighted Dirichlet energy theorem (Theorem \ref{T:Dirichlet}) states that
    \[ 
    G(k) = \frac{1}{k} \sum_{\substack{e \in E\\d(x_0,e) < k}} (2^{p-1})^{d(x_0,e)} \cdot |\nabla u(e)|^p
    \]
    is monotonically non-decreasing in $k$.
We note that if $e=(a_n,a_{n+1})$, then $d(e,0)=n$ and
\begin{align*}
(2^{p-1})^{n} \cdot|\nabla u(e)|^p &= (2^{p-1})^{n} \cdot|a_{n+1} - a_n|^p = (2^{p-1})^{n} \cdot\left[\left(2-\frac{1}{2^n}\right) - \left(2 - \frac{1}{2^{n-1}}\right)\right]^p\\
&= (2^{p-1})^{n} \cdot\frac{1}{2^{pn}}=\frac{1}{2^n}.
\end{align*}

Since $|\{e\in E \ : \ d(e,0)=\ell, \  |\nabla u(e)|\neq 0\}|=2^{\ell+1}$,  
\[
G(k)=\frac{1}{k}\sum_{\ell=0}^{k-1}2^{\ell+1}\frac{1}{2^{\ell}}=\frac{1}{k}2k=2.
\]

Our discrete Weiss monotonicity formula, Theorem \ref{T:Weiss}, proves the monotonicity of
\[
W(k) = \sum_{j=0}^{k-1} 2^j D_j - \frac{H_k}{2^k},
\]
where, if $y_1,y_2$ denote $y$'s children,
\[
D_j=\sum_{y\in V_j}\sum_{\ell=1}^{2}|u(y)-u(y_{\ell})|^2, \qquad H_k=\sum_{y\in V_k}u(y)^2.
\]
We have

\[
\frac{H_k}{2^k} = \frac{1}{2^k}2^{k} \cdot a_k^2=a_k^2 = 4 - \frac{8}{2^k} + \frac{4}{4^k}.
\]
Consequently
\[
W(k) = 2k - \left( 4 - \frac{8}{2^k} + \frac{4}{4^k} \right)= 2k - 4 + \frac{8}{2^k} - \frac{4}{4^k}.
\]

Notice that $W(k+1)-W(k)= 2 - \frac{4}{2^k} + \frac{3}{4^k}>0$, which can be seen letting $t=\frac{1}{2^k}$ and rewriting $W(k+1)-W(k)=2-4t+3t^2>0$.

Regarding the $p$-generalization of Almgren's monotonicity formula, Theorem \ref{T:Almgren}, we have
\[
H(k)=2^k\left(2-\frac{1}{2^{k-1}}\right)^p=2^{k+p}\left(1-\frac{1}{2^k}\right)^p.
\]
Hence
\[
N(k)=\frac{1}{2}H_{k+1}-H_k=2^{k+p}\left[\left(1-2^{-(k+1)}\right)^p-\left(1-2^{-k}\right)^p\right].
\]
One can re-check the monotonicity of $N$ by letting $x=2^{-k}\in (0,1]$ and rewriting 
\[
N(k)=\frac{2^p}{x}\left[\left(1-\frac{x}{2}\right)^p-\left(1-x\right)^p\right].
\]
Letting $h(x)=|1-x|^p$, which is convex for $p\ge 1$, we conclude that
\[
x\mapsto\frac{h(x/2)-h(x)}{x}
\]
is non-increasing in $x\in (0,1)$, which implies $N(k)$ is non-decreasing in $k$.

\subsubsection{Unbounded harmonic functions}\label{SS:unbounded}

Now, consider the infinite 3-regular tree, with root $x_0$. We set $u(x_0)=1$. We set one neighbor to have value $2$, and two neighbors to have value $1/2$. 
We proceed according to the following rule: each vertex $w$ always has three neighbors, two of which have value $u(w)/2$, and one which has value $2u(w)$, see Figure \ref{fig:5}.

\begin{center}
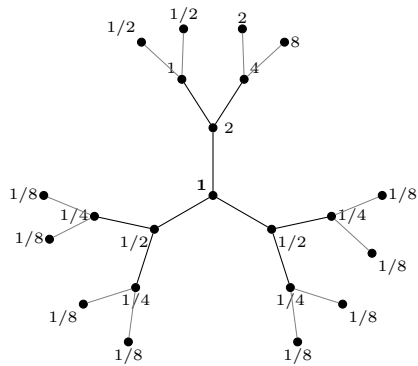
\begin{figure}[htbp]
\begin{tikzpicture}[
scale=0.5,
    dot/.style={circle, draw, fill=black, inner sep=0pt, minimum size=3pt},
    label style/.style={font=\tiny, inner sep=2pt}
]

\node[dot] (R) at (0,0) {};
\node[label style, anchor=south east] at (R) {\textbf{1}};

\node[dot] (N1) at (90:1.8) {};   \node[label style, anchor=west, xshift=2pt] at (N1) {2};
\node[dot] (N2) at (210:1.8) {};  \node[label style, anchor=north east] at (N2) {1/2};
\node[dot] (N3) at (330:1.8) {};  \node[label style, anchor=north west] at (N3) {1/2};
\draw[thin] (R) -- (N1) (R) -- (N2) (R) -- (N3);

\node[dot] (N11) at (75:3.2) {};  \node[label style, anchor=south west] at (N11) {4};
\node[dot] (N12) at (105:3.2) {}; \node[label style, anchor=south east] at (N12) {1};
\draw[thin] (N1) -- (N11) (N1) -- (N12);

\node[dot] (N21) at (190:3.2) {}; \node[label style, anchor=east] at (N21) {1/4};
\node[dot] (N22) at (230:3.2) {}; \node[label style, anchor=north] at (N22) {1/4};
\draw[thin] (N2) -- (N21) (N2) -- (N22);

\node[dot] (N31) at (350:3.2) {}; \node[label style, anchor=west] at (N31) {1/4};
\node[dot] (N32) at (310:3.2) {}; \node[label style, anchor=north] at (N32) {1/4};
\draw[thin] (N3) -- (N31) (N3) -- (N32);

\node[dot] (L3_1) at (65:4.5) {}; \node[label style, anchor=west] at (L3_1) {8};
\node[dot] (L3_2) at (80:4.5) {}; \node[label style, anchor=south] at (L3_2) {2};
\draw[very thin, gray] (N11) -- (L3_1) (N11) -- (L3_2);

\node[dot] (L3_3) at (100:4.5) {}; \node[label style, anchor=south] at (L3_3) {1/2};
\node[dot] (L3_4) at (115:4.5) {}; \node[label style, anchor=south east] at (L3_4) {1/2};
\draw[very thin, gray] (N12) -- (L3_3) (N12) -- (L3_4);

\node[dot] (L3_5) at (180:4.5) {}; \node[label style, anchor=east] at (L3_5) {1/8};
\node[dot] (L3_6) at (195:4.5) {}; \node[label style, anchor=east] at (L3_6) {1/8};
\draw[very thin, gray] (N21) -- (L3_5) (N21) -- (L3_6);

\node[dot] (L3_7) at (220:4.5) {}; \node[label style, anchor=north east] at (L3_7) {1/8};
\node[dot] (L3_8) at (240:4.5) {}; \node[label style, anchor=north] at (L3_8) {1/8};
\draw[very thin, gray] (N22) -- (L3_7) (N22) -- (L3_8);

\node[dot] (L3_9) at (360:4.5) {};  \node[label style, anchor=west] at (L3_9) {1/8};
\node[dot] (L3_10) at (340:4.5) {}; \node[label style, anchor=north west] at (L3_10) {1/8};
\draw[very thin, gray] (N31) -- (L3_9) (N31) -- (L3_10);

\node[dot] (L3_11) at (320:4.5) {}; \node[label style, anchor=north west] at (L3_11) {1/8};
\node[dot] (L3_12) at (300:4.5) {}; \node[label style, anchor=north] at (L3_12) {1/8};
\draw[very thin, gray] (N32) -- (L3_11) (N32) -- (L3_12);

\end{tikzpicture}
\caption{Sketch of the unbounded harmonic function described in \S \ref{SS:unbounded}: each vertex $x$ has 2 neighbors with value $u(x)/2$, and one with value $2u(x)$.}
\label{fig:5}
\end{figure}
\end{center}

To compute the weighted Dirichlet energy, we categorize the edges into two types. 
\begin{itemize}
    \item Type U (upward): from parent $u$ to child $2u$. For these edges, $|\nabla u(e)|^2=(2u - u)^2 = u^2$.
\item Type D (downward): from parent $u$ to child $u/2$. For these edges, $|\nabla u(e)|^2=(u - u/2)^2 = (u/2)^2$.
\end{itemize}

Based on the fact that two neighbors of $x$ will have value $\frac{u(x)}{2}$ and one neighbor will have value $2u(x)$:
\begin{itemize}
    \item a type U edge at level $\ell$ generates: 
    \begin{itemize} 
    \item one type U child. If we call this edge $e_1$, $|\nabla u(e_1)|^2=4u^2$.
   \item  one type D child. If we call this edge $e_2$, $|\nabla u(e_2)|^2=u^2$.
\end{itemize}
\item A type D edge at level $\ell$ generates two type D children. For each of these edges $f_1,f_2$, $|\nabla u(f_j)|^2=(u/4)^2$.
\end{itemize}
Let \[
\mathcal{U}_\ell=2^{\ell}\sum_{\substack{e\in E \text{ type U}\\d(e,x_0)=\ell}}|\nabla u(e)|^2, \qquad \mathcal{D}_{\ell}=2^{\ell}\sum_{\substack{e\in E \text{ type D}\\d(e,x_0)=\ell}}|\nabla u(e)|^2.
\]

  We have $\mathcal{U}_{\ell+1} = 8 \mathcal{U}_\ell$, and since $\mathcal{U}_0 = 1$, we obtain $\mathcal{U}_\ell = 8^\ell$. Moreover,
\[
\mathcal{D}_{\ell+1} = 2 \mathcal{U}_\ell + \mathcal{D}_\ell=2\cdot 8^{\ell} +\mathcal{D}_{\ell}.
\]
Notice that $\mathcal{D}_0 = 2(1/2)^2 = \frac{1}{2}$. We also have 
\[
\mathcal{D}_\ell = \mathcal{D}_0 + \sum_{i=0}^{\ell-1} 2(8^i) = \frac{1}{2} + 2\left(\frac{8^\ell - 1}{7}\right).
\]
We then have
 \[
2^{\ell}\sum_{\substack{(a,b)\in E\\a\in V_{\ell}, b\in V_{\ell+1}}}|u(a)-u(b)|^2  = \mathcal{U}_\ell + \mathcal{D}_\ell = 8^\ell +\frac{1}{2} + 2\left(\frac{8^\ell - 1}{7}\right) = \frac{9 \cdot 8^\ell}{7} + \frac{3}{14}.
 \]

Consequently, 
\[
G(k) = \frac{1}{k} \sum_{\ell=0}^{k-1} \left( \frac{9 \cdot 8^\ell}{7} + \frac{3}{14} \right)=\frac{1}{k} \left[ \frac{9(8^{k} - 1)}{49} + \frac{3k}{14} \right].
\]

Regarding our Weiss monotonicity formula, Theorem \ref{T:Weiss}, we have alredy proven that
\begin{align*}
\sum_{j=0}^{k-1} 2^j D_j=\frac{9(8^k-1)}{49}+\frac{3k}{14}.
\end{align*}
To compute $H_k$, let us categorize the vertices into two types. We say $w\in V_k$ is of   
\begin{itemize}
    \item type A if its parent has value $2u(w)$. 
\item type B if its parent has value $\frac{u(w)}{2}.$
\end{itemize}

Given the structure of this harmonic function, 
\begin{itemize}
\item a type A vertex at level $\ell$ generates two type A children.
    \item a type B vertex at level $\ell$ generates: 
    \begin{itemize} 
    \item one type A child.
   \item  one type B child.
\end{itemize}
\end{itemize}

Let \[
A_\ell=\sum_{v\in V_{\ell} \text{ type A}}u(v)^2, \qquad B_\ell=\sum_{v\in V_{\ell} \text{ type B}}u(v)^2, \qquad H_{\ell}=A_{\ell}+B_{\ell}.
\]
  We have $B_{\ell+1} = 4 B_\ell$, and since $B_1 = 4$, we obtain $B_\ell = 4^\ell$. Moreover, $A_1=\frac{1}{2}$ and
\[
A_{\ell+1} = \frac{A_{\ell}}{2}+ \frac{B_{\ell}}{4}=\frac{A_{\ell}}{2}+4^{\ell-1}.
\]
Therefore
\[
A_{\ell}=\frac{3}{7}\frac{1}{2^{\ell}}+\frac{4^{\ell}}{14}.
\]
Putting these together,
\begin{align*}
W(k)&= \frac{9(8^k - 1)}{49} +\frac{3k}{14} -\frac{1}{2^k}\left(\frac{3}{7}\frac{1}{2^k}+\frac{4^k}{14}+4^k\right).
\end{align*}
One can show that
\[
W(k+1)-W(k)=\frac{9\cdot 8^k}{7}+\frac{3}{14}+\frac{9}{28\cdot 4^k}-\frac{15\cdot 2^k}{14}.
\]
Calling $t=2^k$, we have $8^k=t^3$ and $4^k=t^2$, so for $t\ge 1$
\[
28t^2(W(k+1)-W(k))=36t^5+6t^2+9-30t^3=6t^3(6t^2-5)+6t^2+9> 6t^2+9>0.
\]

To compute $N(k)$, define 
\[
A_\ell^p=\sum_{v\in V_{\ell} \text{ type A}}|u(v)|^p, \qquad B_\ell^p=\sum_{v\in V_{\ell} \text{ type B}}|u(v)|^p, \qquad H_{\ell}^p=A_{\ell}^p+B_{\ell}^p.
\]
When $p=3$, we have $B_{\ell+1} = 8 B_\ell$, and since $B_1 = 8$, we obtain $B_\ell = 8^\ell$. Moreover, $A_1=\frac{1}{4}$ and
\[
A_{\ell+1} = \frac{A_{\ell}}{4}+ \frac{B_{\ell}}{8}=\frac{A_{\ell}}{4}+8^{\ell-1}.
\]
Therefore
\[
A_{\ell}=\frac{15}{31}\frac{1}{4^{\ell}}+\frac{8^{\ell}}{62}.
\]
From this, we obtain $H_{\ell}=\frac{15}{31}\frac{1}{4^{\ell}}+\frac{8^{\ell}}{62}+8^{\ell}$ and
\[
N(\ell)=\frac{1}{2}\left(\frac{15}{31}\frac{1}{4^{\ell+1}}+\frac{8^{\ell+1}}{62}+8^{\ell+1}\right)-\left(\frac{15}{31}\frac{1}{4^{\ell}}+\frac{8^{\ell}}{62}+8^{\ell}\right)=\frac{-105}{248\cdot 4^{\ell}}+\frac{189\cdot 8^{\ell}}{62}.
\]

\subsection{For the $2$-regular tree}

\subsubsection{Bounded harmonic function}

Bounded harmonic functions on the $2$-regular tree are of the form $u(n)=C$. In this case, $G(k)=0$, $W(k)=-\frac{C^2}{k^2}$ and $N(k)=0$.

\subsubsection{Unbounded harmonic function}
Let $u(n)=an+b$. Then $G(k)=\frac{1}{k}2k|a|^p=2|a|^p$, and we have already showed in the proof of Theorem \ref{T:Weiss2d} that $W(k)=a^2-\frac{b^2}{k^2}$. Finally,
\[
N(k)=(|a(k+1)+b|^p+|a(-k-1)+b|^p)-(|ak+b|^p+|a(-k)+b|^p).
\]
Letting $f(k)=|ak+b|^p+|a(-k)+b|^p$, the monotonicity of $N$ is equivalent to $f(k+2)+f(k)\ge 2f(k+1)$. To show this, first notice that for any real numbers $c,d$ we have 
\begin{equation}\label{numberstuff}
|c+d|^p+|c-d|^p\ge 2|c|^p:
\end{equation}
indeed, call $u=c+d$ and $v=c-d$, so $c=\frac{u+v}{2}$. The statement is equivalent to $|u|^p+|v|^p\ge 2\frac{|u-v|^p}{2^p}$, which holds by the power mean inequality, since \[\frac{|u|^p+|v|^p}{2}\ge\left(\frac{|u|+|v|}{2}\right)^p\ge\left(\frac{|u+v|}{2}\right)^p.
\]
Apply \eqref{numberstuff} with $c=a(k+1)+b$, $d=a(-k-1)+b$, so $f(k+1)=|c|^p+|d|^p$, $f(k)=|c-a|^p+|d+a|^p$, $f(k+2)=|c+a|^p+|d-a|^p$. With this notation, $f(k+2)+f(k)\ge 2f(k+1)$ is equivalent to $|c+a|^p+|d-a|^p+|c-a|^p+|d+a|^p\ge 2|c|^p+2|d|^p$, which follows from \eqref{numberstuff}.

\printbibliography
\end{document}